\documentclass[12 pt]{article}
\usepackage{amssymb,amsfonts,amsmath,latexsym,verbatim,amscd,amsthm,graphicx,epsfig}
\title{Multiple-layer solutions to the Allen-Cahn equation \\ on hyperbolic space}
\author{Rafe Mazzeo \thanks{Department of Mathematics, Stanford University, 
Stanford, CA 94305; Email: mazzeo@math.stanford.edu; Supported by the NSF Grant DMS-1105050}
\\ Stanford University \and Mariel Saez \thanks{ Departamento de Matem\'aticas, Avda. Vicu\~na Mackenna 4860. Macul,
 Santiago. Chile; Email: mariel@mat.puc.cl; Supported by Conycit under
 grants Fondecyt de Iniciaci\'on 11070025,  Fondecyt regular 1110048  and proyecto Anillo ACT-125, CAPDE.} \\ P. Universidad
 Cat\'olica de Chile.} 

\newcommand{\be}{\begin{equation}} 
\newcommand{\ee}{\end{equation}}
 
\newcommand{\RR}{{\mathbb R}} 
\newcommand{\HH}{{\mathbb H}}
\newcommand{\rr}{{\mathbb R}} 
\newcommand{\hh}{{\mathbb H}}

\newtheorem{prop}{Proposition}[section]

  \newcommand{\calC}{{\mathcal C}}
  
\newcommand{\calE}{{\mathcal E}}

 \newcommand{\calH}{{\mathcal H}}

 \newcommand{\calO}{{\mathcal O}}

 \newcommand{\calU}{{\mathcal U}}
  \newcommand{\calV}{{\mathcal V}}
\newcommand{\calW}{{\mathcal W}}

\newcommand{\del}{\partial}
\newcommand{\sech}{\operatorname{sech}}
\newcommand{\e}{\epsilon}

\begin{document} 
\bibliographystyle{plain} 

\maketitle
\begin{abstract}
In this paper we study the existence of multiple-layer solutions to the elliptic Allen-Cahn equation in hyperbolic space:
\[
-\Delta_{\hh^n } u+F'(u)=0;
\]
here $F$ is a nonnegative double-well potential with nondegenerate minima. We prove that for any collection of 
widely separated, non-intersecting hyperplanes in $\hh^n$, there is a solution to this equation which has
nodal set very close to this collection of hyperplanes. Unlike the corresponding problem in $\RR^n$,
there are no constraints beyond the separation parameter.
\end{abstract}
\section{Introduction}
Let $f$ be a scalar function on $\RR$ which is the derivative of a nondegenerate double-well potential $F$, i.e.\ 
$F(s) \geq 0$ for all $s \in \RR$; $F^{-1}(0) = \{\pm 1\}$ and $F''(\pm 1) > 0$. 
Associated to this function is the Allen-Cahn equation
\begin{equation}
-\Delta u + f(u) = 0,
\label{ACE}
\end{equation}
which models phase transitions, grain boundaries and other physical and geometric phenomena. As is common
in this subject, we assume that $F$ is monotone decreasing on $(-\infty,-1)$ and monotone increasing on $(1,\infty)$, 
and since its regularity is not particularly germane to our work, we also assume that $F\in \calC^\infty(\RR)$.

This equation has most frequently been studied on domains $\Omega \subset \RR^n$ or else on all of $\RR^n$. 
In these investigations it is standard to consider the family of scaled equations
\begin{equation}
-\Delta u  + \frac{1}{\e^2} f(u) = 0,
\label{ACES}
\end{equation}
especially in the limit as $\e \searrow 0$. Notice that \eqref{ACES} is equivalent to \eqref{ACE} for the equation 
on all of $\RR^n$ by dilation, but these problems are inequivalent on a fixed bounded domain. If $u_\e$ is a solution 
which is energy-minimizing in an appropriate sense, then one is interested in the location of the nodal set $\{u_\e = 0\}$.
Under rather general hypotheses, the limit of this nodal set as $\e \searrow 0$ is a minimal or constant mean curvature
hypersurface, at least in a weak sense.  Another closely related and important aspect of this problem is its relationship
to the De Giorgi conjecture, which asks whether an entire solution to \eqref{ACE} on $\RR^n$ which is monotone in one 
direction necessarily depends on only one variable. The validity of this in low dimensions and its failure in
high dimensions is the analogue of the Bernstein theorem for minimal graphs in $\RR^n$, $n \leq 8$. 
The literature on these problems in Euclidean space is immense, and we cite only \cite{DKW} and references therein.

It is also of interest to study \eqref{ACE} with respect to different ambient geometries, in particular to see how 
curvature properties affect the existence and nature of solutions. Some results in positive curvature were obtained 
in \cite{Sire1}, \cite{Sire2} while the first author and Birindelli \cite{BM} studied these problems in hyperbolic space, 
focusing on the analogue of the De Giorgi conjecture in that setting. Amongst the results obtained there is 
the existence of a unique `one-dimensional' solution $U_0$ of this equation on $\HH^n$, depending only on 
the signed distance from a totally geodesic hyperplane, as well as its uniqueness amongst all bounded solutions 
with the same asymptotic boundary values. The function $U_0$ takes the values $+1$ and $-1$ on two (open) 
hemispheres of $S^{n-1}$, the sphere at infinity; indeed (since the problem is conformally invariant), this also
characterizes solutions which have asymptotic boundary values on any two disjoint open spherical caps,
the union of the closures of which is all of $S^{n-1}$. Once the background metric is no longer Euclidean, 
the study of the problems \eqref{ACE} and \eqref{ACES} are quite different. 

A fairly recent paper by Pisante and Ponsiglione \cite{PP} proves existence of a broad class of solutions to 
\eqref{ACES} in $\HH^n$. (That paper contains a discussion and references concerning motivation 
for studying this problem on hyperbolic space coming from Yang-Mills theory.)  Just as there are 
many more complete properly embedded minimal hypersurfaces in $\HH^n$ than in $\RR^n$, so too 
can one find many solutions of the Allen-Cahn equation in $\HH^n$ with various specified asymptotic boundary
behaviours. One of the results they prove is that if $S \subset S^{n-1}$ is any smooth hypersurface with 
$S^{n-1} \setminus S = \Omega^+ \cup \Omega^-$ the union of two open sets, then there exists a solution
$u_\e$ to \eqref{ACES} with $u_\e \to \pm 1$ on $\Omega^\pm$. They use a barrier method, which is quite
effective but does not control the nodal set of $u_\e$ away from $S^{n-1}$ when $\e > 0$. Their main
interest is in establishing that as $\e \searrow 0$, this nodal set approaches the complete minimal surface with 
asymptotic boundary on $S$. 

Our goal in this paper is to give a quite different construction of  solutions of \eqref{ACE} which allows one
to estimate the nodal set of $u$ rather precisely without taking a limit in $\e$. The solutions we construct
this way are, admittedly, quite limited, but the simplicity of the proof and the possible further uses of the linear
estimates we derive here hopefully make the case that this argument is worth recording.  Let 
$H_1, \ldots, H_k$ denote a disjoint collection of totally geodesic hyperplanes in $\HH^n$, and let
$\del H_j = S_j$; this is a copy of a subsphere $S^{n-2} \subset S^{n-1}$ (this is `round' but not necessarily
totally geodesic so long as $n > 2$, but when $n=2$ it just consists of a pair of points). Write $S^{n-1} \setminus (\bigcup_j S_j) = 
\Omega^\pm$, so each of these open sets is a disjoint collection of open conformal `annuli'. 
Assuming that the mutual distances of the $H_j$ are all sufficiently large, we construct solutions 
of \eqref{ACE} which assume the asymptotic boundary values $\pm 1$ on $\Omega^\pm$; these
solutions have nodal sets which are quantifiably close to the union of hyperplanes $\cup H_j$. 
The method of proof is a rather simple gluing argument where the constituent pieces are the 
one-dimensional solutions of \cite{BM}.  The key step of the proof involves some new linear estimates
for Schr\"odinger-type operators on $\HH^n$ with a `stratified medium' type structure; these estimates
are of independent interest since there are no previous results, to our knowledge, concerning mapping
properties for Schr\"odinger operators on $\HH^n$ where the potential has this type of structure. 

Our result should be contrasted with a gluing result due to del Pino, Kowalczyk, Pacard and Wei \cite{PKPW}
for \eqref{ACE} in $\RR^2$. A key discovery in that paper is that the nodal sets `feel' one another in the
sense that in order to carry out the gluing, the ratios of the distances between these different layers must 
satisfy a nontrivial nonlinear equation which balances the configuration. No such interlayer effects appear in this 
hyperbolic setting. The explanation is simply that the layers are much more clearly separated, particularly 
out near infinity.  This sort of effect is well-known in many geometric and analytic problems in hyperbolic space. 

We anticipate that it is possible to carry out a somewhat more involved gluing argument to establish the
existence of solutions of \eqref{ACE} which have asymptotic boundary values $\pm 1$ on the components 
of a general decomposition $S^{n-1} = \overline{\Omega^+} \cup \overline{\Omega^-}$, where 
$\overline{\Omega^+} \cap \overline{\Omega^-} = S$ is a smooth hypersurface. Although this
is the same result as in \cite{PP}, the gluing method would (as here) give good control of the location of the
nodal set.   A more interesting and new direction is to study the vector-valued analogue of \eqref{ACE} in
$\HH^n$.  A special case is the Ginzburg-Landau equation, where $u$ is valued in $\mathbb C$ and $W = (1-|u|^2)$
is nondegenerate on $\{|z| = 1\}$ in the Morse-Bott sense. The analogues of the one-dimensional solutions from 
\cite{BM} and the interesting classes of asymptotic boundary values on $S^{n-1}$ which these vortices might
attain is not yet known. We hope to return to this elsewhere. 

The authors are grateful to the referee, whose comments were very helpful. 

\section{The single layer problem}
Hyperbolic space $\HH^n$ is a warped product of the real line and a hyperbolic space of one lower dimension, 
$\RR \times \HH^{n-1}$, with metric 
\[
g = dt^2 + \cosh^2 t\, g_{\HH^{n-1}}.
\]
The submanifold $\{t=0\}$ is a totally geodesic hyperplane $H$, and the function $t$ is the signed distance from $H$.
In these coordinates,
\[
\Delta_{\HH^n} = \del_t^2 + (n-1)\tanh t\, \del_t + \sech^2 t\, \Delta_{\HH^{n-1}}.
\]
The paper \cite{BM} contains a construction of a special solution $U_0$ of \eqref{ACE} on $\HH^n$ depending only
on the variable $t$, which thus satisfies the ODE
\begin{equation}
U_0''(t) + (n-1)\tanh t\, U_0'(t) - f(U_0(t)) = 0.
\label{eq:ode}
\end{equation}
In this section we prove mapping properties on certain weighted H\"older spaces for the operator 
\begin{equation}
L_0 = \del_t^2 + (n-1)\tanh t\, \del_t + \sech^2 t\, \Delta_{\HH^{n-1}} - f'(U_0(t)),
\label{eq:aclin}
\end{equation}
which is the linearization of \eqref{ACE} at $U_0(t)$. 

As proved in \cite{BM}, the function $U_0$ is unique amongst solutions of \eqref{eq:ode} which tend to
$-1$ as $t \to -\infty$ and to $+1$ as $t \to +\infty$. We now establish a few more properties about this
solution. First, while it is proved in \cite{BM} that $U_0$ is monotone, i.e. $U_0'(t) \geq 0$ for all $t$, it is in 
fact the case that $U_0'(t) > 0$ for all $t$. Indeed, if this were not the case, i.e.\ if $U_0'(t_0) = 0$ for some $t_0$, 
then $U_0'$ would reach a strict minimum there, so that $U_0''(t_0) = 0$. However, inserting $U_0'(t_0) = U_0''(t_0) = 0$ 
in \eqref{eq:ode} gives $f(U_0(t_0)) = 0$, so that $u(t) \equiv U_0(t_0)$ would be another solution with 
the same Cauchy data at $t_0$, which is impossible. 

Next, let $f'(\pm 1) = \gamma_{\pm}$, and define 
\[
- \beta_\pm = -\frac{n-1}{2} - \sqrt{ \frac{(n-1)^2}{4} + \gamma_\pm}\, .
\]
This is the negative root of the equation $ \lambda^2 + (n-1) \lambda - \gamma_\pm = 0$. Notice
that $-\beta_\pm < -(n-1)$; this will be important later.  Differentiating \eqref{eq:ode} gives 
\[
L_0 U_0'(t) = - (n-1) \sech^2 t  \, U_0'(t) < 0.
\]
On the other hand, for $t \ge T \gg 0$, 
\begin{multline*}
L_0 e^{- \beta_+  t} = ( \beta_+^2 - (n-1) \beta_+ \tanh t - f'(U_0(t)) ) e^{-\beta_+  t}  \\
= (n-1)\beta_+ (1-\tanh t) + (\gamma_+ - f'(U_0(t))) e^{-\beta_+ t} > 0,
\end{multline*}
since $f'(U_0(t)) \nearrow f'(1) = \gamma_+$. This gives that $L_0 ( C e^{-\beta_+ t } - U_0'(t)) > 0$ for $t \geq T$, so if we choose 
$C > 0$ so that $C e^{-\beta_+  T} \geq U_0'(T)$, 
then $U_0'(t) \leq C e^{-\beta_+  t}$ for $t \ge T$.  A similar argument gives $U_0'(t) \leq C e^{-\beta_- |t|}$ for $t \leq -T$. 
Now integrate from $t$ to $\infty$ or $-\infty$ to get that
\[
|1 \mp U_0(t)| \leq C e^{-\beta_\pm |t|} 
\]
as $|t| \to \infty$. Working slightly more carefully, one can even show that 
\[
\begin{array}{rccl}
1 - U_0 & = &  c_+ e^{-\beta_+ t} + \calO(e^{-(\beta_+ + \e)t}), \qquad & t \to \infty, \\
1 + U_0 & = & c_- e^{-\beta_- |t|} + \calO(e^{-(\beta_- + \e)|t|}), & t \to -\infty, 
\end{array}
\]
for some $\e > 0$.  The fact that there are different exponential rates at $\pm \infty$ is often irrelevant below, so we set 
$\beta = \min \{\beta_+, \beta_-\}$ (so $\beta > n-1$), and have proved that $U_0(t) = \pm 1 + \calO(e^{-\beta |t|})$ 
as $t \to \pm \infty$. 

We show finally that $U_0$ is strictly stable on $L^2(\RR; (\cosh t)^{n-1} dt)$, or in other words, the $L^2$ spectrum of 
the operator $-L_0$ is contained in the open half-line $(0,\infty)$. In concrete terms, this means that there exists a constant 
$c > 0$ such that for all $\phi \in \calC^\infty_0(\RR)$, 
\[
\int_{\RR} (-L_0 \phi) \phi \, (\cosh t)^{n-1} dt\geq c \int_{\RR} |\phi|^2\, (\cosh t)^{n-1} dt.
\]
First using $\del_t^2 + (n-1)\tanh t \, \del_t = (\cosh t)^{1-n} \del_t ( (\cosh t)^{n-1} \del_t)$, we have
\[
\int_{\RR} (-L_0 \phi) \phi \, (\cosh t)^{n-1} dt = \int_{\RR} \left(|\del_t \phi|^2 + f'(U_0) |\phi|^2\right) \, (\cosh t)^{n-1} dt.
\]
Since $f'(U_0(t)) \geq c' > 0$ for $|t| \geq T \gg 0$, we see that if $\phi$ is supported in $\{|t| \geq T\}$, 
then the right hand side is bounded from below by $c' \int |\phi|^2 (\cosh t)^{n-1}\, dt$. 
In particular, $\mbox{spec}(-L_0) \cap (-\infty, c')$ is discrete and finite.

Suppose that $\lambda_0$ is the lowest eigenvalue of $-L_0$, and the corresponding eigenfunction is $\phi_0 > 0$.  
Write $\psi = U_0'(t)$ and set $w = \phi_0/\psi$. By a short calculation, 
\[
w'' + \left((n-1)\tanh t + 2\psi^{-1} \psi'\right) w' + (\lambda_0 - (n-1) \sech^2 t) w  = 0.
\]
The same asymptotic analysis from above shows that if $\lambda_0 < 0$, then $|\phi_0(t)| \leq Ce^{-(\beta_\pm + \e) |t|}$ 
for some $\e > 0$ as $t \to \pm \infty$ so that $|w| \to 0$, and this is ruled out by the maximum principle. 

We have now proved that $-L_0 \geq 0$ and there is at most an isolated (and necessarily simple) eigenvalue at $0$.
If this eigenvalue were to exist, then the range of $-L_0$ would be a closed subspace of codimension $1$. Thus
to rule out this zero mode, it suffices to prove that the range of $-L_0$ is dense, and to do this it suffices to show 
that $L_0 u = f$ has a solution $u \in L^2$ for every $f \in \calC^\infty_0$.  For this we use the strictly positive
supersolution $w = U_0'$.  By domain monotonicity, the lowest eigenvalue of $-L_0$ on $[-T,T]$ with Dirichlet
boundary conditions is strictly positive, so for each $T$ there exists a unique $u_T$ which vanishes at the
endpoints and which solves $L_0 u_T = f$ on this interval.  Now choose $C > 0$ so that 
$f \pm C(n-1)\sech^2 t \, \psi \geq 0$.  Then $L_0( u_T - C\psi) \geq 0$ and in addition $u_T-C\psi \leq 0$ at $t = \pm T$,
so $u_T \leq C \psi$ uniformly in $T$. Similarly, $L_0 ( u_T + C\psi) \leq 0$ and $u_T  + C \psi \geq 0$ at $t = \pm T$,
so $u_T \geq -C \psi$. These bounds are uniform in $T$, so we may pass to the limit and obtain a function $u$
on $\RR$ which satisfies $L_0 u = f$ and $|u| \leq C \psi$.  Finally, since $\psi \in L^2( \RR; (\cosh t)^{n-1} dt)$, so
does $u$.   We have now proved that $0$ is not an eigenvalue of $-L_0$, hence $-L_0 \geq c > 0$ as claimed. 

It follows directly that $U_0$ is a strictly stable solution of the PDE \eqref{eq:aclin} on $L^2(\HH^n)$. 
Indeed, by Fubini's theorem, integrating separately in $\RR$ and $\HH^{n-1}$, we obtain
\begin{equation}
\begin{aligned}
& \int_{\HH^n} (-L_0 \phi) \phi\, dV_{\HH^n} =  \\ & \quad \int_{\HH^{n-1}} \int_\RR 
\left( |\del_t \phi|^2 + f'(U_0)|\phi|^2 + \sech^2 t\, |\nabla_{\hh^{n-1}} \phi|^2 \right)\, (\cosh t)^{n-1} \, dt dv_{\hh^{n-1}} \\
& \qquad \qquad \geq c \int_{\hh^{n-1}} \int_\rr |\phi|^2\, (\cosh t)^{n-1}\, dt dv_{\hh^{n-1}}
\end{aligned}
\label{stab}
\end{equation}
simply because $\sech^2 t |\nabla_{\hh^{n-1}} \phi|^2 \geq 0$. 

Altogether we have proved the
\begin{prop}
The operator $-L_0$ is self-adjoint, strictly positive, and invertible on $L^2(\HH^n, dV_{\HH^n})$. 
\label{pr:l2}
\end{prop}

We now study the solvability of $L_0$ acting between the weighted H\"older spaces
\[
\calC^{k,\alpha}_{\mu, \delta}(\HH^n, H) := \sech(\mu t)  \rho^\delta \calC^{k,\alpha}(\HH^n) = \big \{u = \sech(\mu t) \rho^\delta \tilde{u}, \ 
\tilde{u} \in \calC^{k,\alpha}(\HH^n)\big \}.
\]
Here $\calC^{k,\alpha}$ is the ordinary H\"older space on $\HH^n$ (to be concrete, the norm is the supremum
over all balls of radius $1$ of the H\"older norm on those balls defined with respect to the hyperbolic metric), 
the weight function $\sech \mu t$ is self-explanatory, and $\rho$ is a defining function for the boundary
of $\HH^{n-1}$ in the ball model (equivalently, we can use $\sech r$ where $r$ is the distance function in $\HH^{n-1}$
from a fixed point $o$) which is independent of $t$.

We now state and prove the main result of this section.  The situation for $n \geq 3$ is slightly different than for $n=2$, so
we begin with the former. 
\begin{prop}
Suppose that $n \geq 3$; fix any $\delta \in (0, \frac{1}{2}(n-2))$ and $\mu \in (0, \beta)$. Then
\[
L_0 : \calC^{2,\alpha}_{\mu,\delta}(\HH^n, H) \longrightarrow \calC^{0,\alpha}_{\mu, \delta}(\HH^n,H)
\]
is an isomorphism.
\end{prop}
\begin{proof}
It suffices to construct a supersolution $v$ which satisfies $A_1 \sech(\mu t) \rho^\delta \leq 
v  \leq A_1' \sech (\mu t) \rho^\delta$ and $L_0 v \leq -C \sech (\mu t) \rho^\delta$ for some $A_1, A_1', C > 0$. Indeed, 
if we have done this, then we can argue as follows.  Choose an increasing sequence $R_1 < R_2 < \ldots \to \infty$ and let 
$B_j$ denote the ball of radius $R_j$ around some fixed point $o \in \HH^n$.  If $h \in \calC^{0, \alpha}_{\mu, \delta}$, then 
consider the sequence of Dirichlet problems 
\[
L_0 u_j = h\ \mbox{on}\ B_j, \qquad \left. u_j \right|_{\del B_j} = 0.
\]
By Proposition~\ref{pr:l2} and domain monotonicity, there is a unique solution to this equation.  Given the supersolution 
above, choose $A_2 > 0$ sufficiently large so that $|h| \leq A_2 C \sech(\mu t) \rho^\delta$. Then 
\[
L_0 ( u_j + A_2 v) \leq h - A_2 C \sech(\mu  t) \rho^\delta \leq 0,
\]
and $u_j + A_2 v \geq 0$ on $\del B_j$.  Thus $u_j + A_2 v$ is a supersolution for $L_0$ which is
nonnegative on $\del B_j$ so $u_j + A_2 v \geq 0$ on all of $B_j$, or equivalently, 
$u_j \geq -A_2 A_1' \sech(\mu t) \rho^\delta$ independently of $j$. A similar argument gives
an upper bound which is also independent of $j$.  We conclude that $|u_j| \leq C \|h\|_0 \leq C \|h\|_{0,\alpha, \mu, \delta}$ 
with $C$ independent of $j$.  Now use local elliptic regularity, the Arzela-Ascoli theorem and 
a standard diagonalization argument to obtain a subsequence which converges in $\calC^{2,\alpha}$ 
on every compact set and which still satisfies this uniform $\calC^0$ bound. The limit function $u$ 
clearly satisfies the same bound too. 

We turn to the construction of the supersolution.  We first recall that since $\delta < (n-2)/2$, there exists
a function $\phi_\delta > 0$ on $\HH^{n-1}$ which satisfies $\Delta_{\HH^{n-1}} \phi_\delta = - \delta(n-2-\delta) \phi_\delta$
and $\phi_\delta \sim \rho^\delta$ as $\rho \to 0$, see \cite{Ter}.  This generalized eigenfunction (where `generalized' means 
simply that it is not in $L^2$) can be taken to be spherically symmetric with respect to some basepoint $o \in \HH^{n-1}$. 
Now consider the product $U_0'(t) \phi_\delta(z)$; this satisfies 
\[
L_0 (U_0'\phi_\delta)= -((n-1) +\delta (n-2-\delta)) \sech^2 t U_0'\phi_\delta<0.
\label{completess}
\]
However, as $|t|\to \infty$, $U_0'(t) \sim e^{-\beta |t|}$, which decays much faster than $e^{-\mu |t|}$. 
On the other hand, for the stated range of values of $\mu$, $e^{ -\mu |t|}$ is a supersolution of $L_0$
in the region $|t| \geq T$ for $T > 0$ sufficiently large.   

We combine these as follows. First calculate that 
\begin{multline*}  
L_0 ((U_0'+\epsilon)\phi_\delta)=  - \sech^2 t \left( (n-1) +\delta (n-2-\delta) \right) U_0'(t) \phi_\delta  \\ 
- \epsilon \left (f'(U_0)+ \delta (n-2-\delta) \sech^2 t \right) \phi_\delta.
\end{multline*}
Thus given any $T > 0$ we can choose $\e > 0$ sufficiently small so that the right hand side of this is bounded
by $- C \phi_\delta$ in the range $|t| \leq T$.  On the other hand,
\[
L_0 e^{-\mu |t|} = (\mu^2 - (n-1) \mu - \gamma_\pm) e^{-\mu |t|}  + \calO(e^{-\beta T}),\ \ \mbox{when} \  t \geq T,\ \mbox{resp.}\ t \leq -T,
\]
and since $\mu < \beta$, this is bounded by $-C e^{-\mu |t|}$ for $|t| \geq T$.  

Now set
\[
v(t,z)=\min\{ (U_0'(t)+\epsilon)\phi_\delta(z), e^{-\mu |t|} \}.
\]
This satisfies all the required properties.
\end{proof}

The difficulty when $n=2$ is that the functions $\phi_\delta$ are no longer available to us, and indeed, the operator
$\Delta_{\HH^{n-1}}$ is simply the Laplacian on $\RR$, and hence has no nonconstant positive supersolutions. 
Therefore, we establish the corresponding result without requiring any decay in the $\HH^{n-1} = \RR$ direction.
In the following, we let $\calC^{k,\alpha}_{\mu}(\HH^2, H) = \sech(\mu t) \calC^{k,\alpha}(\HH^2)$. 
\begin{prop}
Fix $\mu \in (0, \beta)$. Then 
\[
L_0 : \calC^{2,\alpha}_{\mu}(\HH^2, H) \longrightarrow \calC^{0,\alpha}_{\mu}(\HH^n,H)
\]
is an isomorphism.
\end{prop}
The proof is essentially the same as what was done above, but now using the supersolution $v = \min \{ U_0'(t) + \e, e^{-\mu |t|} \}$.

\section{Multiple layer configurations}\label{2layersol}
We now turn to the construction and analysis of multiple layer approximate solutions for \eqref{ACE}, and the proof that these
may be perturbed to exact solutions. There are three steps to this. First we construct a family of approximate solutions to this
equation, $u_\calH$. Here $u_\calH$ vanishes along a collection $\calH = \{H_1, \ldots, H_N\}$ of disjoint totally geodesic hyperplanes. 
Denoting the minimum of the distances between any pair of elements of this collection by $D_{\calH}$, we then show that the 
linearization $L_\calH$ of \eqref{ACE} at $u_\calH$ is always Fredholm on certain weighted H\"older spaces. If $D_\calH$ is
sufficiently large, then $L_{\calH}$ is invertible on these spaces, and moreover, the norm of its inverse is uniformly bounded
as $D_\calH \to \infty$. From here it is a simple matter to perturb $u_\calH$ to an actual solution of the Allen-Cahn equation.

\subsubsection*{Approximate solutions}
To begin, suppose that $\calH = H_1 \cup \ldots \cup H_N$  be a collection of mutually disjoint totally geodesic hyperplanes 
in $\HH^n$.  Write $\HH^n \setminus \calH = \cup \Omega_j$, where each $\Omega_j$ is a connected open set in $\HH^n$ 
bounded by some number of the $H_j$ and some portion of the sphere at infinity.  We claim that it is possible to `label'
the components $\Omega_j$ with values $+1$ and $-1$ in such a way that if $\Omega_i$ and $\Omega_j$
share some common boundary $H_\ell$, then they have opposite sign.  To prove this, we use induction on $N$.
The claim is obvious when $N = 1$, and if we have established it for some value of $N$ and if $\calH$ 
is a collection of $N+1$ disjoint hyperplanes, then one of the $H_j \in \calH$ is `outermost' in the sense
that all of the remaining $H_i \in \calH$ lie in one component of $\HH^n \setminus H_j$, while the other
component contains no element of $\calH$. Choose a labelling for $\HH^n \setminus (\calH \setminus \{H_j\})$. 
Then exactly one of these components contains $H_j$, and whatever the sign attached to this component,
we assign the opposite sign to the other, `outer', component of $\HH^n \setminus H_j$. 

For each $H_j$, choose a hyperbolic isometry $\varphi_j$ which carries $H_j$ to a fixed totally geodesic hyperplane.  
If $U_0(t)$ is the single-layer solution of \eqref{ACE} associated to $H$, then we let $u_j = \varphi_j^* U_0$; thus
$u_j$ is simply the single layer solution associated to $H_j$. 
Note that there are actually two possibilities for $u_j$, depending on the orientation of  $\varphi_j$, or equivalently,
which side of $H_j$ is carried to which side of $H$. 

To define the approximate solution $u_{\calH}$, choose $u_j$ for each $H_j \in \calH_1$ so that $u_j$ tends to $1$
on the side of $H_j$ which is labelled $+1$ and $u_j$ tends to $-1$ on the side of $H_j$ which is labelled $-1$. 
We then glue together these various $u_j$ as follows. 
For each $j$, define the set $\calV_j$ which consists of the points which are closer to $H_j$ than to any other $H_i$.
The complement $\calE = \HH^n \setminus \cup_j \calV_j$ consists of all points which are equidistant to
at least two of the $H_j$. Let $\calU_j$ be a slight enlargement of $\calV_j$, say the union of all balls of radius
$1$ with centers in $\calV_j$. Thus $\{\calU_1, \ldots, \calU_N\}$ is an open cover of $\HH^n$, and assuming
that $D_\calH \geq 3$, each $\calU_j$ contains exactly one $H_j$. Let $\{\chi_j\}$ be a partition of unity subordinate 
to this open cover, and define $\calW$ to be the set of all points $p$ where $\chi_j(p) > 0$ for more than one $j$. 
Thus $\calW$ is an open neighbourhood around $\calE$.  We finally set
\begin{equation}
u_{\calH} = \sum_{j=1}^N \chi_j u_j. 
\end{equation}
Notice that this function is an exact solution of \eqref{ACE} on $\HH^n \setminus \overline{\calW}$.
As we describe below, the error term, which is its deviation from being an exact solution, is supported
in $\calW$ and is exponentially small as a function of $D_\calH$. 

\subsubsection*{Function spaces}
We next define the weighted H\"older spaces  
\[
\calC^{k,\alpha}_{\mu,\delta} (\HH^n,\calH) := \sech(\mu \tau) \rho^\delta \calC^{k,\alpha}(\HH^n) =
\{ u = \sech(\mu \tau) \rho^\delta \tilde{u}: \tilde{u} \in \calC^{k,\alpha}(\HH^n) \}.
\]
These are exact analogues of the spaces considered in \S 2, and we only need to define the appropriate weight functions 
$\tau$ and $\rho$. The inclusion of $\calH$ in the notation indicates that the weight functions (and hence the spaces themselves)
depend on the configuration. 

The function $\tau$ is defined to be a slight smoothing of the signed distance function from the union of the hyperplanes $H_j$.
Thus $\tau$ has the same sign as $u_\calH$ away from $\calH$, and vanishes on each $H_j$. This
signed distance function is smooth away from the equidistant set $\calE$ defined above. 
(Note that $\calE$ is a union of portions of totally geodesic hyperplanes
meeting at higher codimension totally geodesic subspaces.)  We can mollify this distance function in a small neighbourhood
of $\calE$ to make $\tau$ smooth and satisfy $|\nabla \tau| \leq 1+\e$ everywhere. 

As for the function $\rho$, consider the function $\rho_0$ in the one-layer case. This function is strictly positive on 
$\overline{\HH^n} \setminus (S^{n-1} \cap \overline{H})$, i.e.\ everywhere except where at infinity in $H$. Let $\hat{\chi}$ 
be a smooth nonnegative cutoff function which equals $1$ on a neighbourhood of $\overline{H} \subset 
\overline{\HH^n}$ and which vanishes outside a slightly larger neighbourhood.  Then for the configuration $\calH$ we take 
\[
\rho = \sum_{j=1}^N \varphi_j^*( \hat{\chi} \rho_0 ) + \sum_{j=1}^N \varphi_j^*( 1 - \hat{\chi}). 
\]
This agrees with the pullback $\varphi_j^* (\hat{\chi} \rho_0)$ near $\overline{H_j}$ and is strictly positive elsewhere on
the closure of $\HH^n$. 

\subsubsection*{Analysis of the linearized operator}
Rather than finding sub- and supersolutions again, we describe an alternate parametrix-based method for analyzing the 
linearization $L_\calH = \Delta_{\HH^n} - f'(u_\calH)$ of the nonlinear operator in \eqref{ACE} at $u = u_{\calH}$.  
\begin{prop}
Suppose that the minimal separation $D_\calH$ between elements of $\calH$ is sufficiently large.
Then as a mapping between weighted H\"older spaces, the linearized Allen-Cahn operator 
\[
L_\calH: \calC^{2,\alpha}_{\mu,\delta}(\HH^n, \calH) \longrightarrow \calC^{0,\alpha}_{\mu,\delta}(\HH^n, \calH)
\]
is invertible. Furthermore, the norm of its inverse is uniformly bounded as $D_{\calH} \to \infty$. 
\label{mainlin}
\end{prop}
\begin{proof}
We have proved in \S 2 that if $H \subset \HH^n$ is a totally geodesic hyperplane, and $U_0$ is the one-dimensional solution
which vanishes on $H$, then $L_0 = \Delta_{\HH^n} - f'(U_0)$ is an invertible mapping on these weighted H\"older spaces.
This means that there is a bounded linear operator $G_0: \calC^{0,\alpha}_{\mu,\delta}(\HH^n, H) \rightarrow 
\calC^{2,\alpha}_{\mu,\delta}(\HH^n, H)$ such that $L_0 G_0 = \mbox{Id}$, $G_0 L_0 = \mbox{Id}$. 

As at the beginning of this section, for each $H_j \in \calH$, let $\varphi_j$ be a M\"obius transformation which carries $H_j$ to $H$.
We then pull back $U_0$ to a one-dimensional solution $u_j$ vanishing on $H_j$ and similarly transport the operators $L_0$ 
and $G_0$ to $L_j = \Delta_{\HH^n} - f'(u_j)$ and its inverse $G_j$. 

Recall the open cover $\{\calU_j\}$ of $\HH^n$ and the associated partition of unity $\{\chi_j\}$.  For each $j$, choose 
a smooth nonnegative function $\tilde{\chi}_j$ which takes values in $[0,1]$, such that $\tilde{\chi}_j = 1$ on the 
distance $R$ neighbourhood of the support of $\chi_j$, and such that the support of $\tilde{\chi}_j$ is contained 
in the distance $R+1$ neighbourhood of the support of $\chi_j$, where $R$ is a parameter to be chosen below. We 
also assume that $R  \ll \frac12 D_\calH$, hence the support of $\tilde{\chi}_j$ does not intersect $H_i$
for any $i \neq j$. 
Now define the operator
\[
\tilde{G}_\calH = \sum_{j=1}^N \tilde{\chi}_j G_j \chi_j.
\]
This is the standard way of pasting together local parametrices for an elliptic problem.  We compute:
\[
L_\calH \tilde{G}_\calH = \sum_{j=1}^N \bigg( \tilde{\chi}_j ( \mbox{Id} + (f'(u_j) - f'(u_\calH)) G_j) \chi_j + 
[L_\calH , \tilde{\chi}_j] G_j \chi_j \bigg). 
\]
Since $\tilde{\chi}_j \chi_j = \chi_j$ for each $j$, we can rewrite this as
\[
L_\calH \tilde{G}_\calH = \mbox{Id} + K_\calH,
\]
where the error term decomposes as
\[
K_\calH = K_{\calH}^{(1)} + K_{\calH}^{(2)} = \sum_{j=1}^N \left( \tilde{\chi}_j  (f'(u_j) - f'(u_\calH)) G_j \chi_j\right)
+ \sum_{j=1}^N \left( [L_\calH , \tilde{\chi}_j] G_j \chi_j \right).  
\]

To prove the first part of the theorem, we claim first  that $K_{\calH}$ has small norm when $D_\calH$ is sufficiently large, 
so that $L_\calH$ is in fact invertible. Granting this  claim, we can then invert $\mbox{Id} + K_\calH$ using a Neumann series. 
We remark that with a very slightly more involved parametrix construction, 
one can show that $L_\calH$ is Fredholm for every $\calH$, though the argument below relies on $D_\calH$ 
being large. It is possible that $L_\calH$ may have nullspace for certain special configurations when $D_\calH$ is small.

To prove the claim, we first note that the Schwartz kernel of each summand in $K_{\calH}^{(1)}$, namely 
$ \tilde{\chi}_j(z) (f'(u_j) - f'(u_\calH))(z) G_j(z, \tilde z) \chi_j (\tilde z)$, is supported in the region where $\tilde z \in 
\mbox{supp}\, \chi_j$, $z \in \mbox{supp}\, \tilde{\chi}_j$, and that in this region, $|f'(u_j(z)) - f'(u_\calH(z))| \leq C e^{- \beta 
D_\calH/2}$.  From this it follows that the norm of $K_\calH^{(1)}: 
\calC^{0,\alpha}_{\mu,\delta} \to \calC^{2,\alpha}_{\mu ,\delta}$ is no larger than by $C e^{ (\mu - \beta) D_\calH / 2}$. 
In particular, this is strictly less than $1$ when $D_\calH$ is large. 
Note, however, that even though $G_j$ is smoothing of order $2$, $K_\calH^{(1)}$ is not compact. This is because
the Schwartz kernel $G_j(z, z')$ does not decay as $z, z'$ tend to infinity in a fixed distance neighbourhood of 
$H$ while staying within a bounded distance of one another. This means that $K_\calH^{(1)}$ is not compact. 

Similar considerations apply to $K_\calH^{(2)}$. Indeed, by the choice of $\tilde{\chi}_j$, the commutator 
$[ L_\calH, \tilde{\chi}_j]$ is supported in the region where $\nabla \tilde{\chi}_j$ is nonvanishing, and this 
is disjoint from the support of $\chi_j$. Since the Schwartz kernel $G_j$ is $\calC^\infty$ away from the diagonal, 
each summand in $K_\calH^{(2)}$ has $\calC^\infty$ Schwartz kernel. We can make the norm of this part of the 
remainder term small by making the parameter $R$ sufficiently large. To see this, note that if $h \in \calC^{0,\alpha}_{\mu,\delta}$,
then $\chi_j h \in \calC^{0,\alpha}_{\mu,\delta}$ too, and is supported in the neighbourhood $\calU_j$ which contains
no other $H_i$. This means that $G_j(\chi_j h) \in \calC^{2,\alpha}_{\mu,\delta}(\HH^n, H_j)$.  Furthermore, the
support of $\nabla \tilde{\chi}_j$ has distance $R$ from the support of $\chi_j$, and the function $|\tau|$ is 
approximately $|t_j| - R$ there. Hence in this set, where $\nabla \tilde{\chi}_j \neq 0$, 
\[
\cosh (\mu \tau) \sech(\mu t_j) \leq C e^{-\mu R}.
\]
These observations prove that the norm of $K_\calH^{(2)}$ can be made as small as desired
if we take $R$ large. Since $D_\calH$ is assumed to be large, we have the freedom to make $R$ large, as needed.

We have now proved that the error term $K_\calH$ has norm strictly less than $1$ when $D_\calH \gg 0$. 
In fact, the norm of this error term decreases as $D_\calH \to \infty$.  This means that
\[
G_\calH = \tilde{G}_\calH \circ (\mbox{Id} + K_\calH)^{-1}
\]
is a right (and then necessarily a two-sided) inverse for $L_\calH$, and it is easy to see that its norm is uniformly 
bounded as $D_\calH \to \infty$. 
\end{proof}

\subsubsection*{The contraction mapping argument}
Equipped with the inverse $G_\calH$ constructed above, we can now conclude the argument to prove Theorem~\ref{mainlin}.

For each $\calH$, define $g_\calH = \Delta_{\HH^n} u_\calH - f(u_\calH)$. This function is nonvanishing only on the set 
$\calW$, and satisfies 
\[
|| g_\calH ||_{0,\alpha, \mu, \delta} \leq C e^{-\beta D_\calH/2}. 
\]

To find a function $v \in \calC^{2,\alpha}_{\mu,\delta}$ such that $\Delta_{\HH^n} (u_\calH + v) - f(u_\calH + v) \equiv 0$,
expand the nonlinear term in a Taylor series $f(u_\calH + v) = f(u_\calH) + f'(u_\calH)v + Q(u_\calH, v)$ where the last
term is the quadratic remainder,  and then write 
\[
L_\calH v = -g_\calH - Q(u_\calH, v) \Leftrightarrow  v = - G_\calH ( g_\calH + Q(u_\calH, v) )
\]
It is straightforward to verify that if $||v||_{2,\alpha,\mu,\delta} < \gamma \ll 1$, then
\[
||Q(u_\calH, v)||_{0,\alpha,\mu,\delta} \leq C \gamma^2\quad \mbox{and}\quad ||Q(u_\calH, v_1) - Q(u_\calH, v_2)||_{0,\alpha,\mu,\delta} \leq 
\frac12 ||v_1 - v_2||_{2,\alpha,\mu,\delta}.
\]
Now choose $D_\calH$ so large that $C e^{- \beta >D_\calH} < \gamma$. The standard contraction mapping argument produces 
a unique solution $v$ lying in the ball of radius $\gamma$ in $\calC^{2,\alpha}_{\mu, \nu}$.

In the two-dimensional case, we employ exactly the same argument, noting only that the correction term $v$ need not 
decay at the boundary of each $H_j$.  However, the solution $v$ is very small, and also vanishes on $\del \HH^2$ 
away from the boundaries of all the $H_j$. So in fact, the solution $u$ still has $u^{-1}(0)$ equal to a union of 
curves which are very small perturbations of the various $H_j$, and the boundary of this zero set is the same as 
the boundary of the union of the $H_j$. 

\subsubsection*{Stability of solutions}
\begin{prop}
The solutions $u$ constructed above are strictly stable. In other words, if $L = \Delta_{\HH^n} - f'(u)$, then
there exists a $c > 0$ so that 
\[
\int_{\HH^n} (-L \phi) \phi \, dV_{\HH^n} \geq c \int_{\HH^n} |\phi|^2\, dV_{\HH^n}
\]
for all $\phi \in \calC^\infty_0(\HH^n)$. 
\end{prop}
\begin{proof}
We have already shown that the one-dimensional solution $U_0(t)$ is strictly stable, and that \eqref{stab} holds.
These stability inequalities can be pasted together as follows.  Fix a partition of unity $\{\chi_j^2\}$ similar 
to the one used in the construction of $u_\calH$,  but chosen slightly differently so that the transition region 
take place in the distance $R$ neighbourhood of $\calV_j$ rather than the distance $1$ neighbourhood,
and such that $|\nabla \chi_j^2| \leq C/R$. Here $R$ can be chosen arbitrarily large, so long as $D_\calH$ is 
large enough.  Then for any $\e \in (0,1)$, we can choose $R$ so large that 
\[
\begin{split}
c \int |\phi|^2 &= c \sum_j \int \chi_j^2 |\phi|^2  \leq \sum_j \int \left(|\nabla (\chi_j \phi)|^2  + f'(u_\calH) \chi_j^2 |\phi|^2\right) \\
& = \sum_j \int \left(\chi_j^2 |\nabla \phi|^2 + 2 \chi_j \phi \nabla \chi_j \cdot \nabla \phi + \phi^2 |\nabla \chi_j|^2 + f'(u_\calH)
\chi_j^2 |\phi|^2\right)  \\
& \leq 2 \int \left(|\nabla \phi|^2  + f'(u_\calH)|\phi|^2\right)  + C \e \int |\phi|^2,
\end{split}
\]
where $C$ is independent of $\e$ and $R$, $u_\calH$ is the approximate solution and $u_j$ is the single layer solution $U_0$ which
equals $u_\calH$ in the support of $\chi_j$. Finally, writing $u = u_\calH + v$ where $v$ is small, we have that 
$|f'(u_j) - f'(u)| \leq \e$ in the support of $\chi_j$ as well. Thus if $\e$ is sufficiently small, we can absorb these terms 
into the other side. 
\end{proof}

\end{document}